\documentclass[notitlepage]{revtex4-1}
\pdfoutput=1
\usepackage[utf8]{inputenc}
\usepackage{parskip}
\usepackage{fullpage}
\usepackage{amsmath,amssymb,mathtools,comment}
\usepackage{subfig}
\usepackage{amsthm}
\newtheoremstyle{mytheorem}
  {8pt} 
  {6pt} 
  {} 
  {} 
  {\bfseries} 
  {.} 
  {.5em} 
  {} 
\theoremstyle{mytheorem}
\newtheorem{theorem}{Theorem}[section]
\theoremstyle{mytheorem}

\theoremstyle{mytheorem}

\theoremstyle{mytheorem}
\newtheorem*{remark}{Remark}
\theoremstyle{mytheorem}
\newtheorem{corollary}{Corollary}[section]

\usepackage{natbib}
\usepackage{hyperref}

\usepackage{tikz}
\usetikzlibrary{matrix}

\begin{document}

\title{Sub-Riemannian cubics in SU(2)}
\author{Michael Swaddle}
\email{meswaddle@protonmail.com}
\author{Lyle Noakes}
\email{lyle.noakes@uwa.edu.au}
\affiliation{Faculty of Engineering, Computing and Mathematics, The University of Western Australia, Crawley 6009, Australia}
\date{\today}

\begin{abstract}
    \noindent \textit{Sub-Riemannian cubics} are a generalisation of Riemannian cubics to a sub-Riemannian manifold. Cubics are curves which minimise the integral of the norm of the covariant acceleration. Sub-Riemannian cubics are cubics which are restricted to move in a horizontal subspace of the tangent space. When the sub-Riemannian manifold is also a Lie group, sub-Riemannian cubics correspond to what we call a sub-Riemannian Lie quadratic in the Lie algebra. The present article studies sub-Riemannian Lie quadratics in the case of \( \mathfrak{su}(2)\), focusing on the long term dynamics.
\end{abstract}
\maketitle

\section{Sub-Riemannian cubics}
    Let \(G\) be a matrix Lie group with a positive-definite bi-invariant inner product \( \langle, \rangle_B \) on the Lie algebra \( \mathfrak{g} := TG_e\). Choose a positive-definite self-adjoint operator \( \mathcal{J} \) with respect to \( \langle , \rangle_B\). Now define \( \langle , \rangle_\mathcal{J} \), by \( \langle X ,Y \rangle_{\mathcal{J}} = \langle X , \mathcal{J} Y \rangle_B \). Given a basis \( \hat{e}_1, \dots \hat{e}_n \) for \( \mathfrak{g}\), define an \( n \times n\) matrix \(\mathcal{J}_{ij} \) by \( \mathcal{J}_{ij} = \langle \hat{e}_i, \hat{e}_j \rangle_\mathcal{J}\). Then given \( V =  v_k  \hat{e}_k \in \mathfrak{g}, \, v_k \in \mathbb{R} \), we can compute \( \mathcal{J} V = v_i \mathcal{J}_{ij} \hat{e}_j\), where repeated indicies are summed.

    A left-invariant metric \( \langle , \rangle\) is defined on \( TG\) by the formula \( \langle X,Y\rangle_g := \langle g^{-1} X, g^{-1} Y \rangle_\mathcal{J} \). Then given a vector subspace of \( \delta_1 \subset \mathfrak{g}\), we define a left-invariant \textit{distribution} on \(G\) to be the vector sub-bundle \( \delta \) of \( TG\), whose fibre \( \delta_g\) over \(g\) is \( g \delta_1\).

   Previous work \cite{cubics1,cubics2,cubics3,null-cubics,noakes-ratiu-euler,duality-and-cubics} has investigated the critical points of the functional
	\begin{equation}
        S\left[\widetilde{x}\right] := \int_0^T \langle \nabla_t\dot{\widetilde{x}}, \nabla_t \dot{\widetilde{x}} \rangle_{\hat{x}} \, dt ,
    \end{equation}
    where \( \widetilde{x} : \left[0,T\right] \rightarrow M \), and \( \widetilde{x}(0) ,\widetilde{x}(T),  \dot{ \widetilde{x}}(0)\) and \( \dot{\widetilde{x}}(T)\) are given.  \( \nabla_t \dot{\widetilde{x}}\) denotes the covariant acceleration, and \(M\) is a Riemannian manifold. In this situation critical points of \( S\) are called \textit{Riemannian cubics}. We now consider the case where, \(M=G\) and \( \dot{\widetilde{x}} \) is constrained to be in the distribution \( \delta\). With this constraint, we will call critical points of \(S\) a \textit{sub-Riemannian cubic}.

    Note that restricting the original Riemannian metric to the distribution makes \(G\) a sub-Riemannian manifold \cite{montgomery}. When the metric is not bi-invariant, \( \mathcal{J} \neq I\), where \(I\) is the identity matrix, we need an underlying Riemannian metric to define \( \nabla_t \dot{\widetilde{x}} \), which is not necessarily restricted to the distribution.

    The equations for normal sub-Riemannian cubics can be derived from the Pontryagin Maximum Principle (PMP). For a reference on the PMP see \cite{pmp-book}. Usually the PMP applies for control systems on \( \mathbb{R}^n\) but there is a version for control systems on a Lie group \cite{opt-control-lie-group}.

    As \(G\) is left-invariant \( \dot{\widetilde{x}} \) is constrained by the equation
	\begin{align}
		\dot{\widetilde{x} } := \widetilde{x} \, \widetilde{V},
	\end{align}
    where \( \widetilde{V} : \left[0,T\right] \rightarrow  \delta_1 \). We also require \(\delta_1 \) to be a bracket generating subset of \( \mathfrak{g}\).

    Let left Lie reduction by \( \widetilde{x} \) be denoted \(L^{-1}\). Then the left Lie reduction of the covariant acceleration \( \nabla_t \dot{\widetilde{x}}\) to \( \mathfrak{g}\) includes a first order derivative of \( \widetilde{V} \),
	\begin{equation}
        L^{-1}(\nabla_t \dot{\widetilde{x}}) = \dot{\widetilde{V}} - \mathcal{J}^{-1}\left[ \mathcal{J} \widetilde{V},\widetilde{V}\right].
    \end{equation}

    To use the PMP define a new control function \( u:\left[0,T\right] \rightarrow \mathfrak{g}\), and treat \( \widetilde{V}\) as an additional state variable. So to  minimise \( S[\widetilde{x}]\) subject to the constraints
	\begin{align}
		\dot{\widetilde{x} } &= \widetilde{x} \, \widetilde{V}, \\
		\dot{\widetilde{V}} &= u,
	\end{align}
    we form the PMP \textit{Hamiltonian}, \(H\),  given by
	\begin{align}
        H(\widetilde{x}, \widetilde{V},u, \lambda,\vartheta, \nu)  := \lambda(\widetilde{x} \,\widetilde{V}) +  \vartheta(u) + \frac{\nu}{2} \left\Vert u - \mathcal{J}^{-1} \left[ \mathcal{J} \widetilde{V}, \widetilde{V}\right] \right\Vert^2_\mathcal{J},
	\end{align}
    where \( \lambda \in T^{*}G_{\widetilde{x}}\) and \( \vartheta \in \mathfrak{g}^* \) are the \textit{co-states}, and \( \nu \leq 0 \). Then the PMP says maximising \(H\) for all \(t\) is a necessary condition for minimising \( S\).

	By the PMP, the co-states are required to satisfy
	\begin{align*}
		\dot{\lambda}(y) &= -dH_{\widetilde{x}}(y) = -\lambda(y \widetilde{V} ),\\
        \dot{\vartheta}(Z) &= -dH_{\widetilde{V}}(Z) =- \lambda(\widetilde{x} Z) - \nu \left\langle- \mathcal{J}^{-1}\left[ \mathcal{J} Z, \widetilde{V}\right] + \mathcal{J}^{-1}\left[ \mathcal{J}\widetilde{V},Z \right], u - \mathcal{J}^{-1} \left[ \mathcal{J} \widetilde{V}, \widetilde{V}\right] \right\rangle_\mathcal{J}
	\end{align*}
    \( \forall y \in TG_{\widetilde{x}} \) and \( Z \in \mathfrak{g} \). \( \lambda \in T^{*}G_{\widetilde{x}}\) can be associated with a \( \lambda^* \in \mathfrak{g}^{*}\) via left multiplication, \( \Lambda^*(V) = \lambda(L(\widetilde{x}) V)\). Differentiating \( \Lambda^*\),
	\begin{align*}
		\dot{\Lambda^*}(y) &= \dot{\lambda}(\widetilde{x}\,y) + \lambda(\dot{\widetilde{x}}\,y)\\
				   &= - \lambda(\widetilde{x}\, y\, \widetilde{V} ) + \lambda( \widetilde{x} \,\widetilde{V} y)\\
				   &= -\Lambda^*(y \,\widetilde{V} ) + \Lambda^*( \widetilde{V} y )\\
				   &= \Lambda^*\left(\left[ \widetilde{V},y\right] \right).
	\end{align*}
    Finally  \( \Lambda^* \) can be associated with \( \Lambda \in \mathfrak{g} \) via the bi-invariant inner product, \( \langle, \rangle_B\)
	\begin{equation*}
        \Lambda^*(y) = \langle \Lambda, y \rangle_B.
	\end{equation*}
	Likewise for \( \vartheta\) there is an associated \( \Theta \in \mathfrak{g} \).
	\begin{equation*}
		\vartheta(w) = \langle \Theta, w \rangle_B.
	\end{equation*}
	This gives the following equations for the costates
	\begin{align*}
		\dot{\Lambda} &= \left[ \Lambda , \widetilde{V} \right],\\
		\dot{\Theta} &=  -\Lambda - \nu \bigg( \mathcal{J}\left[\widetilde{V},u\right] - \left[ \mathcal{J} \widetilde{V},u \right]
\\
&+ \mathcal{J} \left[ \widetilde{V},\mathcal{J}^{-1} \left[ \mathcal{J}\widetilde{V},\widetilde{V} \right] \right] - \left[ \mathcal{J}\widetilde{V},\mathcal{J}^{-1} \left[ \mathcal{J}\widetilde{V},\widetilde{V} \right] \right]\bigg).
	\end{align*}
	By the PMP there are two cases to consider.

	\subsubsection{Normal case}
        In the normal case, \( \nu < 0 \), the optimal control \( u^*\) must maximise \(H\). From now on we consider the case when \( \langle,\rangle_{\mathcal{J}}\) is simply bi-invariant and so \( \mathcal{J} = I\). We now use the notation \( \langle, \rangle :=\langle, \rangle_\mathcal{J} = \langle, \rangle_B \). Without loss of generality set \( \nu = -1\). Hence in the normal case, the PMP Hamiltonian can equivalently written as
		\begin{equation*}
			H = \big\langle \Lambda , \widetilde{V} \big\rangle + \big\langle \Theta , u \big\rangle - \frac{1}{2} \big\langle u , u \big\rangle.
		\end{equation*}
		Maxima occur when \( d H(u^*)_u = 0\), so
		\begin{equation}
		 \big\langle \Theta, u^* \big\rangle - \big\langle u, u^* \big\rangle  = 0.
		\end{equation}
		Therefore optimal controls occur when \( \mathrm{proj}_{\delta_1} ( \Theta) = u = \dot{\widetilde{V}}\). The equations for the costates reduce to a single equation, which gives the following theorem.

        \begin{theorem}
            \( \widetilde{x}\) is a normal sub-Riemannian cubic if and only if
        \begin{equation}
			\ddot{\Theta} = \left[ \dot{\Theta}, \widetilde{V} \right].
			\label{equation:normal-cubic}
		\end{equation}
        \end{theorem}

        \begin{remark}
            Denote the projection of \( \Theta\) onto the orthogonal complement, \( \perp\), of \( \delta_1\) by \( \varphi  = \mathrm{proj}_{\perp}(\Theta) \). Let \( \delta_1 = \mathfrak{g}\). Then the resulting equations for normal sub-Riemannian cubics match the bi-invariant Riemannian case, \( \dddot{\widetilde{V}} = \left[ \ddot{\widetilde{V}}, \widetilde{V} \right] \) \cite{noakes-ratiu-euler}. In general solutions to (\ref{equation:normal-cubic}) are hard to find.
        \end{remark}

        \begin{remark}
            One subclass of solutions are the so called \textit{linear Lie quadratics}. In this case, \( \varphi = \varphi_0\) and \( \widetilde{V} = ( q_0 + q_1 t + q_2 t^2 ) \widetilde{V}_0\), where \( q_0 , q_1, q_2  \in \mathbb{R}\), \( \widetilde{V}_0\) is a constant in \( \delta_1\), and \( \varphi_0\) is a constant in the orthogonal complement of \( \delta_1\).
    	\end{remark}

    \(\dot{\Theta}\) can be found in terms of \( \widetilde{x} \). Rewrite (\ref{equation:normal-cubic}), take the adjoint and integrate,
	\begin{align*}
	\mathrm{Ad}(\widetilde{x}) \left( \ddot{\Theta} + [ \widetilde{V}, \dot{\Theta}] \right)  = 0 \\
        \implies \dot{\Theta} = \mathrm{Ad}( \widetilde{x}^{-1} )\mathcal{A}
	\end{align*}
    where \( \mathcal{A} \in \mathfrak{g} \). This simply reflects the fact that \( \dot{\Theta} \) satisfies a \textit{Lax equation} \cite{noakes-lax-constraints} and is therefore isospectral.

    \subsubsection{Abnormal case}
	    The abnormal case is given by \( \nu = 0\). As before, the PMP Hamiltonian can be written as
		\begin{equation*}
			H = \big\langle \Lambda , \widetilde{V} \big\rangle + \big\langle \Theta , u \big\rangle.
		\end{equation*}
		Maxima occur when \( dH(u^*)_u = 0\). Immediately this requires \( \mathrm{proj}_{\delta_1}( \Theta) = 0\), so there is no way to determine \(u\) from the PMP.

    \subsubsection{Bounds}
        Given some function \(f\), we say that \(f \) is \(\mathcal{O}(t^n)\), when for some \( c > 0\), \( \left| f \right| \leq c \left| t^n \right| \), for all \( t \in \mathbb{R}\).

        \begin{corollary}
        \begin{align}
			\big\langle \ddot{\widetilde{V}},\widetilde{V} \big\rangle - \frac{1}{2}  \big\langle \dot{\widetilde{V}} ,  \dot{\widetilde{V}} \big\rangle &= c_1 , \label{equation:cons1}\\
			\big\langle  \ddot{\widetilde{V}} ,\ddot{\widetilde{V}} \big\rangle + \big\langle \dot{\varphi}, \dot{\varphi} \big\rangle &= c_2, \label{equation:cons2}
		\end{align}
		where \( c_1 \in \mathbb{R} \) and \( c_2 \geq 0  \).
        \end{corollary}

        \begin{proof}
         First take the inner product of (\ref{equation:normal-cubic}) and \( \widetilde{V} \) to find \( \langle \dddot{\widetilde{V}},\widetilde{V} \rangle = 0 \). Next take the inner product of (\ref{equation:normal-cubic}) with \( \ddot{\widetilde{V}} + \dot{\varphi}\) to find \( \langle \dddot{ \widetilde{V}}, \ddot{\widetilde{V}} \rangle + \langle \dot{\varphi}, \dot{\varphi} \rangle = 0\). Integrating these gives the result.
        \end{proof}

        \begin{corollary}
            \[ \frac{1}{2} c_1 t^2 + c_5 t + c_6 \leq \big\langle \widetilde{V},\widetilde{V} \big\rangle \leq  \mathcal{O}(t^4). \]
        \end{corollary}

        \begin{proof}
            As \( \big\langle \dot{\varphi}, \dot{\varphi} \big\rangle \geq 0\), we have \( \big\langle \ddot{\widetilde{V}} ,\ddot{\widetilde{V}} \big\rangle \leq c_2 \). This argument can be repeated for the components of \( \ddot{\widetilde{V}} \), so given \( \widetilde{V} =  v_k(t) \widehat{e}_k\), then  \(\ddot{v}_k^2 \big\langle \hat{e}_k, \hat{e}_k \big\rangle  \leq c_2\). As we are working with the bi-invariant metric,  we have \( - \sqrt{c_2} \leq \ddot{v_k} \leq \sqrt{c_2} \) and so \( | v_k | \leq \frac{1}{2} \sqrt{c_2} t^2 + c_3 t + c_4 \), where \( c_3 \) and \(c_4\) are some other constants. Therefore \( \big\langle  \widetilde{V} ,\widetilde{V} \big\rangle  \) is bounded above by  \( \mathcal{O}(t^4)\).

		    The same argument can be used to to show \(\big\langle \dot{\varphi}, \dot{\varphi} \big\rangle\) is bounded above a constant, and then \( \left\Vert \varphi \right\Vert \) is bounded above by a linear function. This then shows \(\big\langle \varphi, \varphi \big\rangle \) is bounded above by \( \mathcal{O}(t^2)\).

		    Equation (\ref{equation:cons1}) can be written as
	    	\begin{align}
			\frac{d^2}{dt^2} \left( \big\langle \widetilde{V},\widetilde{V} \big\rangle \right) = c_1 + \frac{3}{2} \big\langle \dot{\widetilde{V}},\dot{\widetilde{V}} \big\rangle.
 	    	\label{equation:cons3}
            \end{align}
 	    	Immediately this yields the lower bound
		\(  \big\langle \widetilde{V},\widetilde{V} \big\rangle \geq \frac{1}{2} c_1 t^2 + c_5 t + c_6 \) where \( c_5\) and \( c_6 \) are other constants.
    \end{proof}

\section{Sub-Riemannian Lie quadratics and symmetric pairs}
	\subsection{Symmetric pairs}
    Let \( (\mathfrak{g}, \mathfrak{h} )\) be a symmetric pair, namely \( \mathfrak{g} = \mathfrak{m} + \mathfrak{h} \) where the following properties hold
	\begin{align*}
		\left[ \mathfrak{m}, \mathfrak{m} \right] &\subseteq \mathfrak{h}, \\
		\left[ \mathfrak{h}, \mathfrak{m} \right] &\subseteq \mathfrak{m}, \\
		\left[ \mathfrak{h}, \mathfrak{h} \right] &\subseteq \mathfrak{m}.
	\end{align*}
    An example is \( \mathfrak{su}(2) \) where \( \mathfrak{m} \) is spanned by the Pauli matrices \( i \sigma_1\) and \( i \sigma_2\), and \(\mathfrak{h}\) is spanned by \( i \sigma_3\). Suppose we set \( \delta_1 = \mathfrak{m}\). The equations for normal sub-Riemannian cubics in \( \mathfrak{g} \) separate into two components.
	\begin{align*}
		\ddot{\varphi} &= \left[\ddot{\widetilde{V}},\widetilde{V}\right],\\
		\dddot{\widetilde{V}} &= \left[ \dot{\varphi}, \widetilde{V}\right].
	\end{align*}
	Integrating the first equation and substituting leaves
	\begin{align}
		\dddot{\widetilde{V}} &= \left[ \left[\dot{\widetilde{V}},\widetilde{V}\right],\widetilde{V}\right] + \left[ \widetilde{\mathcal{C}},\widetilde{V}\right].
	\label{equation:sym-cubic}
	\end{align}
    We call \( \widetilde{V}\) which satisfy this equation a \textit{sub-Riemannian Lie quadratic}.
	One simple solution to this is \( \widetilde{V} = \widetilde{V}_0 + \widetilde{V}_1 t \), where \( \widetilde{V}_0 \) and  \( \widetilde{V}_1\) are chosen so \( \left[\widetilde{V}_0 ,\widetilde{V}_1 \right] = \widetilde{C} \). We call sub-Riemannian Lie quadratics \textit{null} when \( \widetilde{\mathcal{C}}=0 \).

	\subsection{Duality}

    We say \(\widetilde{W}\) is \textit{dual} to \(\widetilde{V}\), when
	\begin{align}
		\widetilde{W} &= -\mathrm{Ad}(\widetilde{x}) \widetilde{V},\\
		\dot{\widetilde{y}} &= \widetilde{y} \, \widetilde{W},
	\end{align}
	where \( \widetilde{y}  = \widetilde{x}^{-1} \).

    In the null case, \( \widetilde{\mathcal{C}} = 0 \), \( \widetilde{V} \) is dual to a non-null \textit{Riemannian Lie quadratic}, \(V\), which is defined by the equation
    \[ \ddot{V} = [ \dot{V}, V]  + C.\]
    Duality was considered for Riemannian Lie quadratics in \cite{duality-and-cubics}. We investigate the sub-Riemannian case.
     \begin{theorem}
    \( \widetilde{V}\) is dual to a rescaled non-null Riemannian Lie quadratic.
    \end{theorem}

    \begin{proof}
    Recall that for any other function \(Z\)
	\begin{equation}
		\frac{d}{dt}\left(\mathrm{Ad}\left(\widetilde{x}\right)Z\right) =\mathrm{Ad}(\widetilde{x}) \left( \dot{Z} + \left[\widetilde{V}, Z\right] \right).
	\end{equation}
	Computing derivatives,
	\begin{align*}
		\dot{\widetilde{W}} &= -\mathrm{Ad}(\widetilde{x} ) \left( \dot{\widetilde{V}} \right),\\
		\ddot{\widetilde{W} } &= -\mathrm{Ad}(\widetilde{x} ) \left( \ddot{\widetilde{V}} + \left[ \widetilde{V} , \dot{\widetilde{V}} \right] \right), \\
		\dddot{\widetilde{W} } &= -\mathrm{Ad}(\widetilde{x} ) \left( \dddot{\widetilde{V}} + \left[\widetilde{V},\ddot{\widetilde{V}}\right] + \left[ \widetilde{V}, \ddot{\widetilde{V}} + \left[ \widetilde{V}, \dot{\widetilde{V}} \right] \right] \right),
	\end{align*}
	which gives
	\begin{align*}
		\dddot{\widetilde{W}  } &= -2 \mathrm{Ad}(\widetilde{x} ) \left(  \left[ \widetilde{V} , \ddot{\widetilde{V}}\right] + \left[ \widetilde{V}, \left[ \widetilde{V} ,\dot{\widetilde{V}} \right] \right] \right).
	\end{align*}
	Then
	\begin{align*}
		\left[ \widetilde{W} ,\ddot{\widetilde{W}  }\right] = \mathrm{Ad}(\widetilde{x} ) \left( \left[ \widetilde{V}, \ddot{\widetilde{V}} \right] + \left[\widetilde{V}, \left[\widetilde{V},\dot{\widetilde{V}} \right] \right] \right) ,
	\end{align*}
	and so
	\begin{align*}
		\dddot{\widetilde{W}  } = 2 \left[ \ddot{\widetilde{W}  } , \widetilde{W}  \right].
	\end{align*}
	Integrating this equation leaves a reparameterised non-null Riemannian Lie quadratic
	\begin{align}
		\ddot{ \widetilde{W} } = 2 \left[ \dot{\widetilde{W} },\widetilde{W} \right] + \widetilde{\mathcal{D}}.
	\end{align}
	Without loss of generality, let \( \widetilde{x}(0) = I\). Consider the equations for \( \widetilde{W} \) and \( \widetilde{V}\) at \( t = 0\). Clearly \( \widetilde{V}(0) = \widetilde{W} (0) \), \( \dot{\widetilde{V}}(0) = \dot{\widetilde{W} }(0)\), and \( \ddot{\widetilde{W} }(0) = \ddot{\widetilde{V}}(0) + \left[ \widetilde{V}(0), \dot{\widetilde{V}}(0)\right] \). First this shows \( \widetilde{W} \) is no longer constrained to \( \delta_1\). Additionally we must have \( \widetilde{\mathcal{D}} = \ddot{\widetilde{V}}(0) + 3 \left[ \widetilde{V}(0), \dot{\widetilde{V}}(0)\right]  \).

	Now let \( V(t) = a \widetilde{W}(b\, t) \), where \( a , b \in \mathbb{R}\). Setting \( a = 2 b\), and \( \mathcal{C} = 2 b^3 \widetilde{\mathcal{D} }\), \( V\) satisfies
	\begin{align}
		\ddot{V} = \left[ \dot{V},V \right] + \mathcal{C},
	\end{align}
	which is the equation for a non null Riemannian Lie quadratic.
    \end{proof}

    Note if we define \( x:[0,T] \rightarrow G\), and \( \dot{x} = x V\), then there is no clear relation between \(x\), a non-null Riemannian cubic and \( \widetilde{x}\). If \( \mathcal{C} = 0 \) then the Lie quadratic is called \textit{null} \cite{null-cubics}. It is possible to integrate the equation for \(V\) in certain groups.  This can occur in a non-trivial way if \( \ddot{\widetilde{V}}(0) = 0 \) and \( \dot{\widetilde{V}}(0) = 0\) or \( \widetilde{V}(0) = 0\). Let
	\[ W = - \mathrm{Ad}(x) V.\]
	Computing derivatives, we find
	\[ \ddot{W} = 0,\]
    which gives \( W = W_0 + W_1 t\), where the \(W_k\) are constant matrices. If \(V\) was known it would be possible to work backwards and compute \(\widetilde{x}\) using the work of \cite{noakes-lax-constraints}.

	\subsection{SU(2)}
        Let \(G = \mathrm{SU}(2)\). Take \( \delta_1 = \mathrm{span}\left( \{ \hat{e}_1 , \hat{e}_2 \} \right) = \mathrm{span}\left( \frac{i}{\sqrt{2}} \{ \sigma_1, \sigma_2 \} \right) \), where \(\sigma_i\) are the Pauli matrices. Let \( \widetilde{\mathcal{C}} = C\frac{i}{\sqrt{2}} \sigma_3\), where \( C \in \mathbb{R}\). Recall \( \mathfrak{su}(2)\) can be identified with \( \mathfrak{so}(3)\). \( \mathfrak{so}(3)\) can then be identified with Euclidean three space, \(\mathbb{E}^3\), with the cross product. As a consquence of the vector triple product formula, we can write for \( \widetilde{V} \) in \( \mathfrak{su}(2) \)
	\begin{align}
		\dddot{\widetilde{V}} = 2 \left( \left\langle \widetilde{V}, \dot{\widetilde{V}} \right\rangle \widetilde{V} - \left\langle \widetilde{V},\widetilde{V} \right\rangle \dot{\widetilde{V}} \right) + [ \widetilde{\mathcal{C}},\widetilde{V}].
	\end{align}

    We can identify \( \widetilde{V} = v_1 (t) \hat{e}_1 + v_2(t) \hat{e}_2 \) with a \( v\in \mathbb{C} \) by taking \( v := v_1 + i v_2\). Then the sub-Riemannian cubic equation in \(\delta_1\) can be written as
	   \[ \dddot{v} = \frac{1}{2} v \left( \dot{\bar{v}} v - \dot{v} \bar{v} \right) - i C v. \]
     Assuming \( v(t) \neq 0\) for all t, define \( \omega : \mathbb{R} \rightarrow \mathbb{S}^1 \subset \mathbb{E}^2 \cong \mathbb{C}\) by
     \[ \omega(t) := \frac{v(t)}{\Vert v(t)\Vert}.\]
     Define \( \mathrm{exp}(y) = ( \cos(y), \sin(y) ) \equiv e^{iy} \in \mathbb{C} \). Choose a \( \vartheta(t_0) \in[0, 2 \pi )\) so that \( \mathrm{exp}(\vartheta(t_0)) = \omega(t_0)\). Then there is a unique continuous function \( \vartheta : \mathbb{R} \rightarrow \mathbb{R}\) such that the diagram
    \begin{center}
    \begin{tikzpicture}
      \matrix (m) [matrix of math nodes,row sep=3em,column sep=4em,minimum width=2em] {
                    &  \mathbb{R} \\
         \mathbb{R} & \mathbb{S}^1 \\};
      \path[-stealth]
        (m-2-1) edge node [below] {$\vartheta$} (m-1-2)
        (m-2-1.east|-m-2-2) edge node [below] {$\omega$} (m-2-2)
        (m-1-2) edge node [right] {$\mathrm{exp}$} (m-2-2)
                edge [dashed,-] (m-2-1);
  \end{tikzpicture}
  \end{center}
  commutes. Then we have \( v(t) = q(t) ( \cos (\vartheta(t)), \sin (\vartheta (t))) = q e^{ i \vartheta}\), where \( q(t) = \Vert v(t) \Vert = \Vert \widetilde{V} \Vert \).

  Substituting back, and taking the real and imaginary components gives the two equations
	\begin{align}
		\dddot{q} - 3 \dot{q} \, \dot{\vartheta}^2 - 3 q\, \ddot{\vartheta} \dot{\vartheta}  &= 0,\\
		q \dddot{\vartheta} + 3 \dot{q}\, \ddot{\vartheta}  + 3 \ddot{q}\, \dot{\vartheta} - q \,\dot{\vartheta}^3 +  q^3\, \dot{\vartheta} + C\, q &= 0.
        \label{equation:complex}
	\end{align}
	Multiplying the first equation by \( r\), and integrating leaves
    \[ -\frac{1}{2} \dot{q}^2 + \ddot{q} q - \frac{3}{2} q^2 \dot{\vartheta}^2  +c_1=0,\]
    where \( c_1 \in \mathbb{R} \). Note that equation follows directly from equation (\ref{equation:cons3}), but we use the complex structure to show several additional properties.

	\subsection{ \(c_1 > 0\) }
    \label{subsection:su2}
    Let \( c_1 > 0 \). Recall \( \langle V , V \rangle \) is at most \( \mathcal{O}(t^4)\). Therefore \( q \) increases no faster than \( \mathcal{O}(t^2)\). Likewise \( \dot{q}^2 + q^2 \dot{\vartheta}^2\) must not increase faster than \(\mathcal{O}(t^2)\).

    Also note \( q^2 \dot{\vartheta}^2 \leq c_2 + c_1 \). Recall \(q^2 \) was bounded below by a quadratic and above by \( \mathcal{O}(t^4)\), so at most \( \dot{\vartheta}^2 = \mathcal{O}(t^{-2} )\). Additionally \(q^2 \dot{\vartheta}^2\) is at most \( \mathcal{O}(t^{-2})\). Therefore,
	\[ \frac{-1}{2} \dot{q}^2 + \ddot{q}q = c_1 + \mathcal{O}(t^{-2} ). \]
    First set \( Y = \dot{q}^2\). Then it follows
    \[  - Y + \frac{d Y}{dq } q = 2 c_1 + \mathcal{O}(t^{-2}), \]
    so
    \[ \frac{d}{dq} \left( \frac{Y}{q} \right) = \frac{2 c_1 +\mathcal{O}(t^{-2})}{q^2}.\]
    Integrating with respect to \(q\)
    \[ \dot{q}^2 = c_7 q - 2 c_1  + \mathcal{O}(t^{-2}), \]
    where \( c_7 \geq 0\). Up to \( \mathcal{O}(t^{-2})\) error, and as \( q\) is at most \( \mathcal{O}(t^2) \), we can write
    \[ \dot{q} = \left( c_7 q - 2 c_1 \right)^{ \frac{1}{2}} + \left( c_7 q - 2 c_1 \right)^{ - \frac{1}{2} } \mathcal{O}(t^{-2} ). \]
    Then
    \begin{align*}
     \frac{\dot{q}}{c_7 q - 2 c_1} = 1 + \mathcal{O}(t^{-2}).
    \end{align*}
    Integrating with respect to \(t\)
    \[ \frac{2 ( c_7 q - 2 c_1)^{\frac{1}{2}}}{c_7} = t + \frac{c_8}{c_7} + \mathcal{O}(t^{-1}), \]
    which gives
    \[ q = \frac{ ( c_7 t + c_8)^2}{4 c_7} +\frac{ 2 c_1}{ c_7} + \mathcal{O}(t^{-1} ).\]

	\section{Asymptotics}
    In \(G = \mathrm{SU}(2) \), and with \( c_1 >0\), it is possible to show that long term asymptotes exist. For Riemannian cubics in \(\mathrm{SO}(3)\), it was established that a limit
    exists in \cite{null-cubics}. In \(\mathrm{SU}(2)\), we can show that the limit
     \[ \alpha_{\pm}(\widetilde{V} ) = \lim_{t\rightarrow  \pm\infty} \frac{ \widetilde{V}}{ \left\Vert \widetilde{V} \right\Vert}, \]
    exists. Using the (smooth) identification of \( \delta_1 \) with \( \mathbb{C} \), we can equivalently show \( \vartheta\) tends to a constant
    \[ \alpha_{\pm} = \lim_{t \rightarrow \pm \infty} \vartheta,\]
    recalling the definition of \( \vartheta\) from the previos section.

    \begin{theorem}
    \( \alpha_{\pm}(\widetilde{V} ) = \lim_{t\rightarrow \pm\infty} \widetilde{V} / \left\Vert \widetilde{V} \right\Vert \) exists.
    \end{theorem}

    \begin{proof}
    Note that we only need to consider \( t \rightarrow \infty\), as \( \widetilde{V}\) can be re-parameterised.
    Using results from section (\ref{subsection:su2}), \( \dot{\vartheta} \) behaves at most like \( \mathcal{O}(t^{-2}) \),
    \[ | \dot{\vartheta} | \leq \frac{1}{d_1 t^2},\]
        as \( q \) is bounded by a quadratic. First note that for \( s \geq r \)
	\begin{align*}
		\left\vert \vartheta(s) - \vartheta(r)  \right\vert &\leq \int^s_r \vert \dot{\vartheta} \vert \, dt \\
        & \leq \frac{1}{d_1} \left( \frac{1}{r} - \frac{1}{s} \right).
	\end{align*}
    We can show that for an unbounded sequence of increasing times, \( t_1, t_2, \dots \), the sequence \( \vartheta(t_1), \vartheta(t_2), \dots \) converges to a limit,
    \[ L = \lim_{n\rightarrow \infty} \vartheta(t_n). \]
    Given some \( \varepsilon \geq 0\), there exists a \(N\) such that for all \(n, m \geq N \), where \( n \geq m\), we have
    \begin{align*}
        \left\vert \vartheta(t_n) - \vartheta(t_m) \right\vert &\leq \frac{1}{d_1} \left( \frac{1}{t_m} - \frac{1}{t_n} \right) \\
        &\leq \frac{1}{d_1}  \left( \frac{1}{t_m} \right)  \\
        &\leq \frac{\varepsilon}{2}.
    \end{align*}
    by choosing \( t_N \geq (1 + \frac{2 }{ d_1 \varepsilon} )  \) and so the sequence is Cauchy. As \( \vartheta \) is a real function, by completeness of \( \mathbb{R}\), the sequence converges. So given \( \varepsilon > 0\), there exists an \(N\) such that for all \( n > N\)
    \[\left\vert L - \vartheta(t_n) \right\vert \leq \frac{\varepsilon}{2}.\]
    Using a similar argument as before, given an \(\varepsilon \geq 0\), there exists a \(T\) such that for \( s , t \geq T\),
    \[  \left\vert  \vartheta(s) - \vartheta(t) \right\vert \leq \frac{\varepsilon}{2} \]
    Now choose \(T = t_N\) and for \( t \geq T\) we have
    \[ \left\vert L - \vartheta(T) \right\vert \leq \frac{\varepsilon}{2},\]
    and
    \[ \left\vert \vartheta(T) - \vartheta(t) \right\vert \leq \frac{\varepsilon}{2}. \]
    By the triangle inequality this gives
    \[ \left\vert L - \vartheta(t) \right\vert \leq \varepsilon. \]
    Therefore
    \[ L = \lim_{t\rightarrow \infty} \vartheta.\]
     Hence take \( \alpha_{+} = L\). Likewise \( \alpha_{-}\) exists by reparameterising.
    \end{proof}
    Using the identification this shows \( \alpha_{\pm}(\widetilde{V}) \) exists.

    For null Riemannian cubics a similar limit was found in Theorem (5) of \cite{asymptotics-e3}. A similar approach can be used to establish a more precise statement on the convergence when the sub-Riemannian cubic is null. Define
    \begin{align}
        \mathcal{V} = \widetilde{V} + \frac{1}{ 2\left\Vert \widetilde{V} \right\Vert^2 } \ddot{\widetilde{V}}.
    \end{align}
    Recall \( q(t) = \left\Vert \widetilde{V}(t) \right\Vert \).

    \begin{theorem}
        If \( \widetilde{\mathcal{C}} = 0 \), \( \left\Vert \frac{\mathcal{V}}{q} - \alpha_{\pm}(\widetilde{V} ) \right\Vert \leq \frac{\sqrt{c_2}}{2 q^3}\).
    \end{theorem}

    \begin{proof}
    Again, considering \( t \rightarrow \infty\) as the negative case can be found via re-parameterisation.
    \begin{align*}
        \frac{d}{dt} \left(  \frac{ \mathcal{V} }{q }\right) &= \frac{ \dot{\mathcal{V}}}{q} - \frac{ \dot{q} \mathcal{V} }{q^2} \\
                                                             &= \frac{1}{q} \left(  \dot{\widetilde{V}} -\frac{\dot{q}}{q^3} \ddot{\widetilde{V}} + \frac{1}{2 q^2} \left( 2 \left( \dot{q} q \widetilde{V} - q^2 \dot{\widetilde{V}} \right) + [ \widetilde{\mathcal{C}},\widetilde{V}]\right) \right) - \frac{\dot{q}}{q^2} \left(\widetilde{V} + \frac{1}{2q^2} \ddot{\widetilde{V}} \right)\\
                                                             &= -\frac{3 \dot{q} \ddot{\widetilde{V}} }{2 q^4} + \frac{ [\widetilde{\mathcal{C}},\widetilde{V}]}{2 q^3}.
    \end{align*}
        Then  if \(\widetilde{\mathcal{C}} = 0\), and noting that \( q > 0 \) , and for large enough \(t\), \( \dot{q} > 0\), assuming \(c_7 > 0 \) and \( c_1  > 0\),
    \begin{align*}
        \left\Vert \frac{ \mathcal{V}(s) }{q(s)  } - \frac{ \mathcal{V}(r) }{q(r) } \right\Vert &\leq \left\Vert \int^s_r  \frac{d}{dt} \left(  \frac{ \mathcal{V}(t) }{q(t)  }\right) \, dt \right\Vert \\
        &\leq  \int_r^s  \frac{ 3 \dot{q} \sqrt{c_2}}{2 q^4} \, dt\\
        &\leq   \frac{\sqrt{c_2}}{2 q(r)^3} -  \frac{\sqrt{c_2}}{ 2 q(s)^3} \\
        &\leq \frac{  \sqrt{c_2} }{2 q(r)^3}.
    \end{align*}
    As \( \ddot{\widetilde{V}} \) is bounded,
    \begin{align*}
        \lim_{t\rightarrow\infty} \frac{\mathcal{V} }{q} = \lim_{t\rightarrow\infty} \frac{\widetilde{V}}{q} = \alpha_{+}(\widetilde{V}).
    \end{align*}
        So finally, taking the limit as \( s \rightarrow \infty \)
    \begin{align}
        \left\Vert \alpha_{+}(\widetilde{V} ) -  \frac{\mathcal{V}(r) }{q(r)} \right\Vert \leq \frac{\sqrt{c_2}}{2 q(r)^3}.
        \label{equation:limit1}
    \end{align}
    \end{proof}

    \begin{corollary}
    \begin{equation}
        \left\Vert  q\, \alpha_{\pm}(\widetilde{V} ) - \widetilde{V} \right\Vert \leq \frac{\sqrt{c_2}}{2 q^2}.
    \end{equation}
    \end{corollary}
    \begin{proof}
        Multiply equation (\ref{equation:limit1}) through by \(q\).
    \end{proof}
    Recall from section (\ref{subsection:su2}), in the long term limit, assuming \( c_1,c_7 > 0\), \(q\) approaches a quadratic.  Hence in the long term \(\widetilde{V} = q \, \alpha_{+}(\widetilde{V})+ \mathcal{O}(t^{-4})  \)

    In the non-null case, a different estimate can be made for \( \widetilde{V}\).
    \begin{align*}
         \frac{ \mathcal{V}(s) }{q(s)  } - \frac{ \mathcal{V}(r) }{q(r) } - \int_r^s \frac{ [\widetilde{\mathcal{C}},\widetilde{V}]}{2 q^3} \, dt&=  -\int_r^s \frac{3 \dot{q} \ddot{\widetilde{V}} }{2 q^4}  \, dt.
    \end{align*}
    As before, taking norms, and letting \( s \rightarrow \infty\),
    \begin{align*}
        \left\Vert  \alpha_{+}(\widetilde{V} )  - \frac{ \mathcal{V}(r) }{q(r) } - \int_r^\infty  \frac{ [\widetilde{\mathcal{C}},\widetilde{V}]}{2 q^3}  dt \right\Vert  &\leq \frac{\sqrt{c_2}}{2 q(r)^3}.
    \end{align*}
    Multiplying through by \(q\), we can deduce
    \begin{align}
        \widetilde{V}(r) = q(r) \alpha_{+}(\widetilde{V} ) - q(r) \left[ \widetilde{\mathcal{C}}, \int_r^\infty \frac{\widetilde{V}}{2 q^3}\, dt \right] +\mathcal{O}(r^{-4}).
        \label{equation:vest}
    \end{align}
    With this we can estimate \(V(r)\) recursively up to \( \mathcal{O}(r^{-4})\) error.

    \begin{theorem}
    \begin{align*}
        \widetilde{V}(r) = q(r) \alpha_{+}(\widetilde{V}) - q(r) \int_{r}^{\infty}  \frac{1}{2 q(t)^2} \, dt \, \left[ \widetilde{\mathcal{C}} , \alpha_{+}(\widetilde{V}) \right]   + \mathcal{O}(r^{-4}).
    \end{align*}
    \end{theorem}
    \begin{proof}
        By equation (\ref{equation:vest}), substituting \( \widetilde{V}(r)\) back
        \begin{align*}
            \widetilde{V}(r) &= q(r) \alpha_{+}(\widetilde{V}) - q(r) \int_{r}^{\infty}  \frac{1}{2 q(t)^2} \, dt \, \left[ \widetilde{\mathcal{C}} , \alpha_{+}(\widetilde{V}) \right]\\
            &+  q(r)\int_r^\infty \frac{[\widetilde{\mathcal{C}},\mathcal{O}(t^{-4})]}{2 q(t)^2} dt  + q(r) \int_r^{\infty} \frac{ \left[ \widetilde{\mathcal{C}},\ \left[ \widetilde{\mathcal{C}}, \int_t^\infty \frac{\widetilde{V}(s)}{2 q(s)^3}\, ds \right] \right] }{2 q(t)^2} dt + \mathcal{O}(r^{-4}).
        \end{align*}
        We should ignore terms smaller than \( \mathcal{O}(t^{-4})\). Also recall \( \widetilde{V}\) behaves like \( \mathcal{O}(t^2)\) for large \(t\).
        \begin{align*}
            \widetilde{V}(r) &= q(r) \alpha_{+}(\widetilde{V}) - q(r) \int_{r}^{\infty} \frac{1}{2 q(t)^2} \, dt \, \left[ \widetilde{\mathcal{C}} , \alpha_{+}(\widetilde{V}) \right]\\
            &+  q(r)\int_r^\infty \frac{[\widetilde{\mathcal{C}},\mathcal{O}(t^{-4})]}{2 q(t)^3} dt  + q(r) \int_r^{\infty} \frac{ \left[ \widetilde{\mathcal{C}},\ \left[ \widetilde{\mathcal{C}}, \mathcal{O}(t^{-3}) \right] \right] }{2 q(t)^2} dt + \mathcal{O}(r^{-4})\\
            \widetilde{V}(r) &= q(r) \alpha_{+}(\widetilde{V}) - q(r) \int_{r}^{\infty} \frac{1}{2 q(t)^2} \, dt \, \left[ \widetilde{\mathcal{C}} , \alpha_{+}(\widetilde{V}) \right]\\
            &+  q(r)\mathcal{O}(r^{-9}) +  q(r) \mathcal{O}(r^{-6})+ \mathcal{O}(r^{-4})\\
            \widetilde{V}(r) &= q(r) \alpha_{+}(\widetilde{V}) - q(r) \int_{r}^{\infty} \frac{1}{2 q(t)^2} \, dt \, \left[ \widetilde{\mathcal{C}} , \alpha_{+}(\widetilde{V}) \right] + \mathcal{O}(r^{-4}).
        \end{align*}
    \end{proof}

    \section*{Example 1}
        Equations (\ref{equation:complex}) can be numerically solved with \textit{Mathematica}'s \texttt{NDSolve} function for the components \( (q,\vartheta)\) of \(v\). Figure (\ref{figure:para1}) is a parametric plot of \(v_1 = q \cos(\vartheta)\) vs \(v_2 = q \sin(\vartheta )\), where \(C=1\), \(v_1(0)= 4, v_2(0)= -1.75, \dot{v}_1(0)= -0.1, \dot{v}_2(0) = 2.5, \ddot{v}_1(0)= -5\) and \(\ddot{v}_2(0)= -5 \).

Initially we see oscillation before stabilising in the long term. Figure (\ref{figure:comp1}) shows the radial and angular components of \(v\). Note how \(q\) approaches a quadratic, \( \dot{q}\) approaches a linear function, and \( \vartheta\) approaches a constant as discussed in the previous sections.

Finally the equation for \(\widetilde{x}\) can also be numerically integrated, using the previously found \(v\). In \( \mathrm{SU}(2)\), \(\widetilde{x}\) is a matrix with four components which satisfy \( \widetilde{x}_{11}^2 +\widetilde{x}_{12}^2 + \widetilde{x}_{21}^2 + \widetilde{x}_{22}^2 = 1\), which is the sphere \( \mathbb{S}^3\). Figure (\ref{figure:x1}) shows a stereographic projection of the components of \( \widetilde{x}\) onto \( \mathbb{R}^3\),
        via
        \[  \left( \widetilde{x}_{11}, \widetilde{x}_{12}, \widetilde{x}_{21} , \widetilde{x}_{22} \right) \rightarrow \frac{1}{1-\widetilde{x}_{22}} \left( \widetilde{x}_{11} ,\widetilde{x}_{12} ,\widetilde{x}_{21} \right). \]
        \begin{figure}[h]
            \centering
            \includegraphics[width=0.3\textwidth]{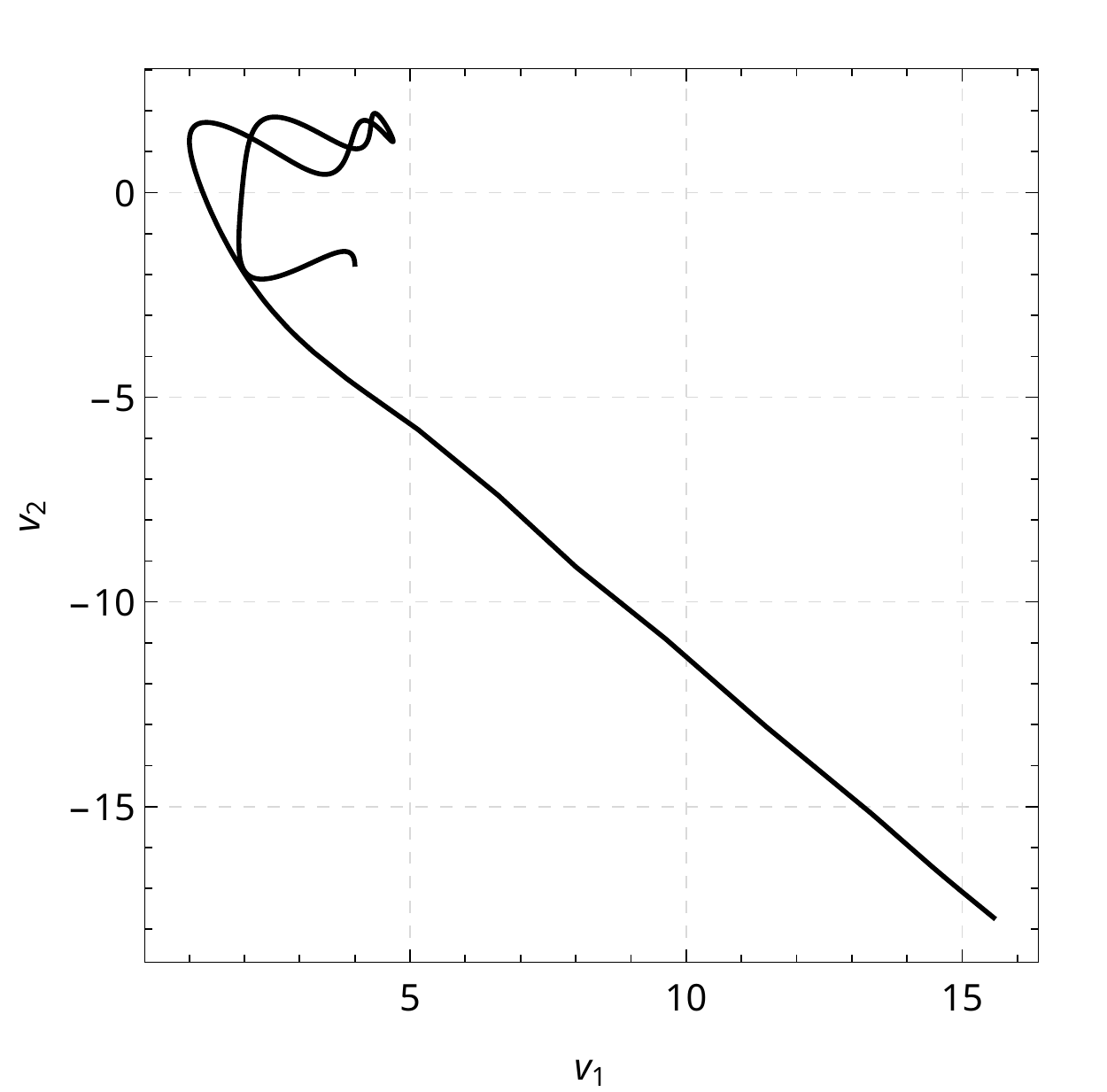}
            \caption{Parametric plot of \(v\) per example 1.}
            \label{figure:para1}
        \end{figure}

        \begin{figure}[h]
            \subfloat[]{\includegraphics[width=0.45\textwidth]{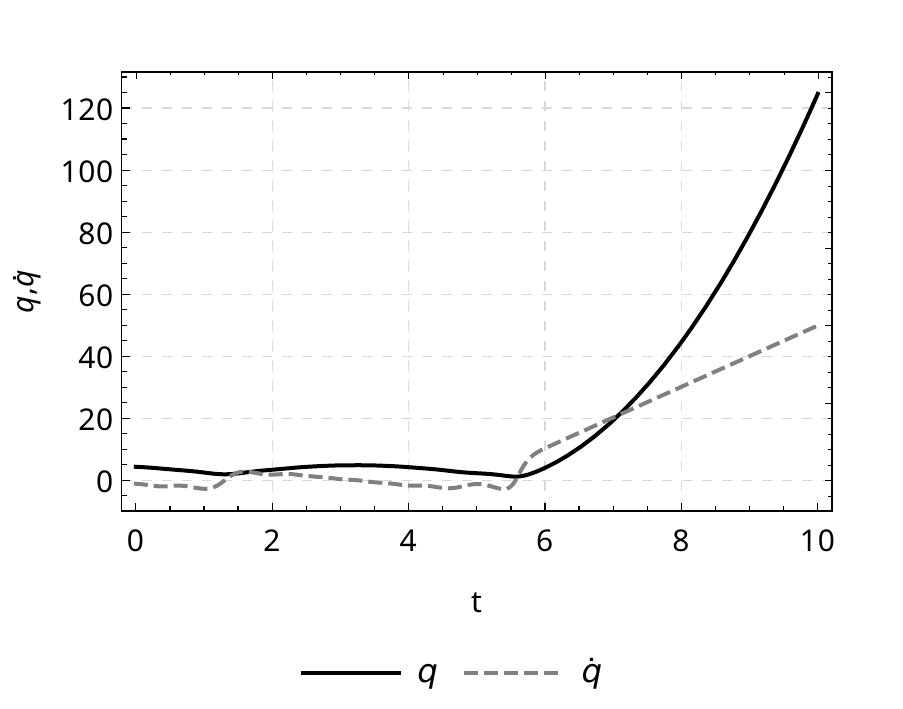}}
            \subfloat[]{\includegraphics[width=0.45\textwidth]{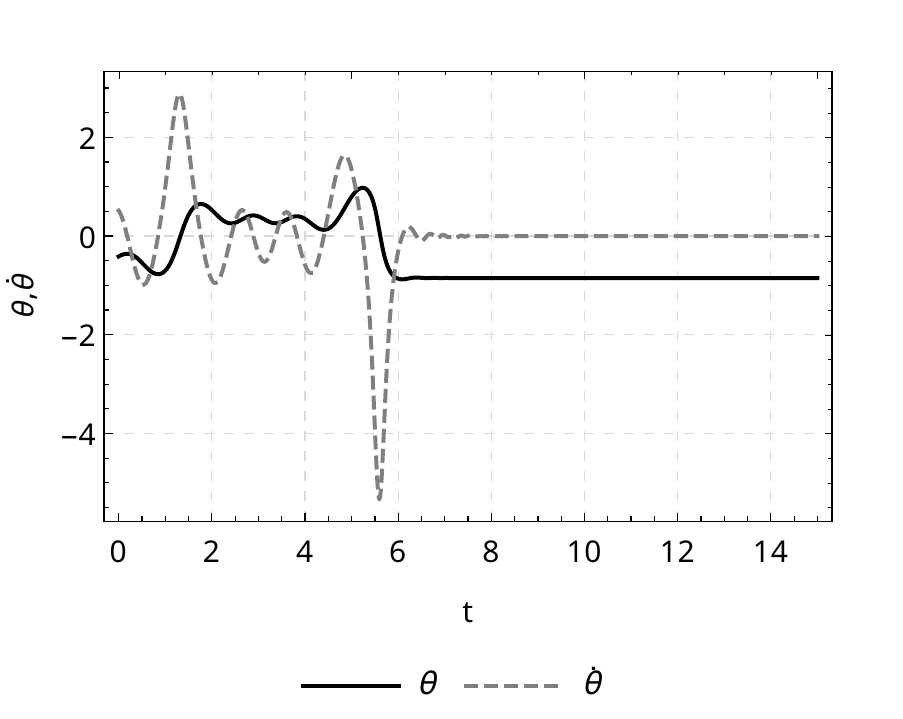}}
            \caption{Radial and angular components of \(v\) per example 1.}
            \label{figure:comp1}
        \end{figure}

        \begin{figure}[h]
            \centering
            \includegraphics[width=0.3\textwidth]{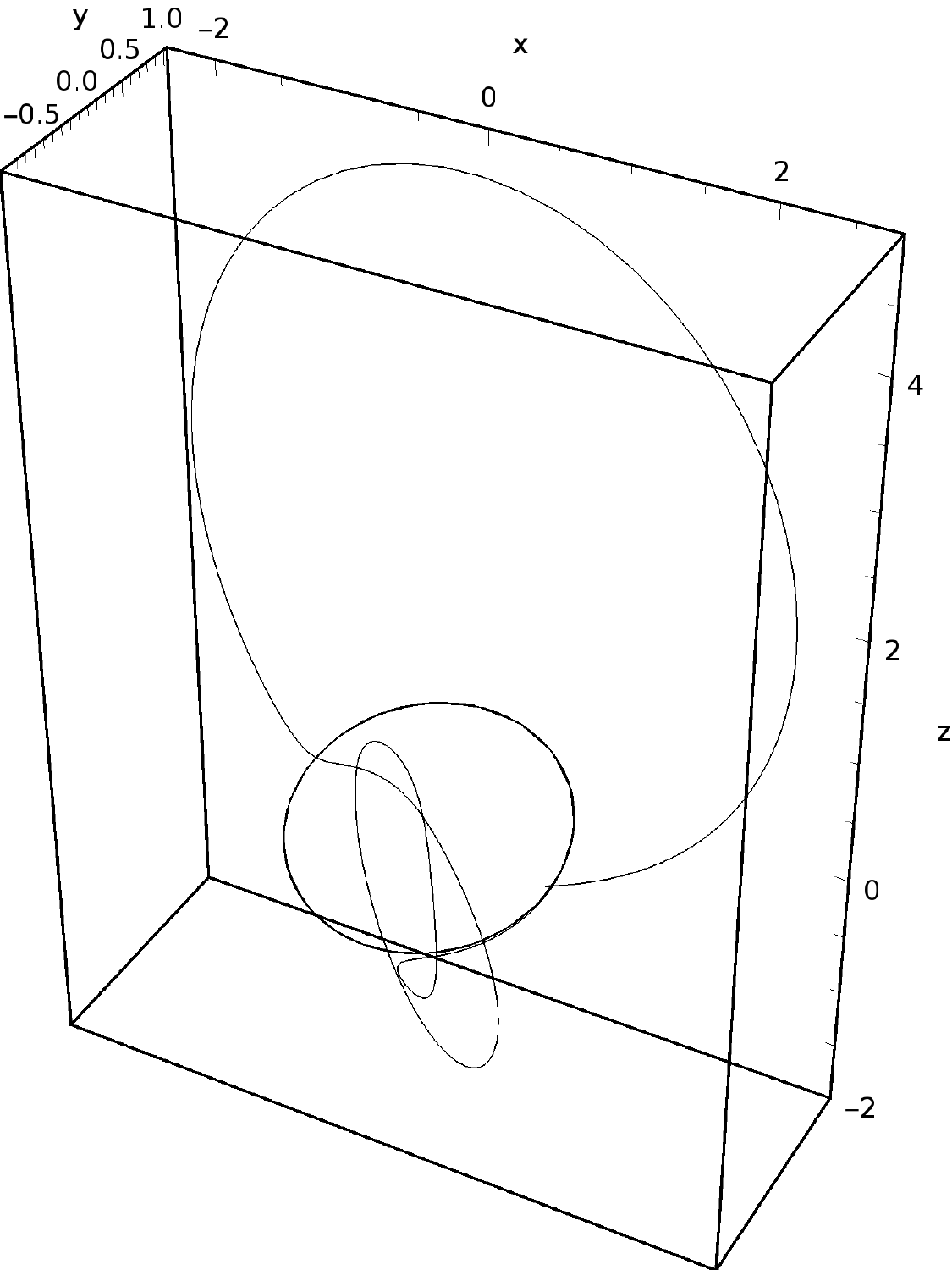}
            \caption{Components of \( \widetilde{x} \) projected into \( \mathbb{R}^3\) per example 1.}
            \label{figure:x1}
        \end{figure}
    \cleardoublepage
    \section*{Example 2}
    Setting \( C=0\) can yield just as interesting dynamics as \( C \neq 0\).  The following figures show \(v\) as per equations (\ref{equation:complex}) with \(v_1(0)= 2, v_2(0)= -1, \dot{v}_1(0)= 2, \dot{v}_2(0) = -1, \ddot{v}_1(0)= 0\) and \(\ddot{v}_2(0)= 5 \).
         \begin{figure}[h]
            \centering
            \includegraphics[width=0.3\textwidth]{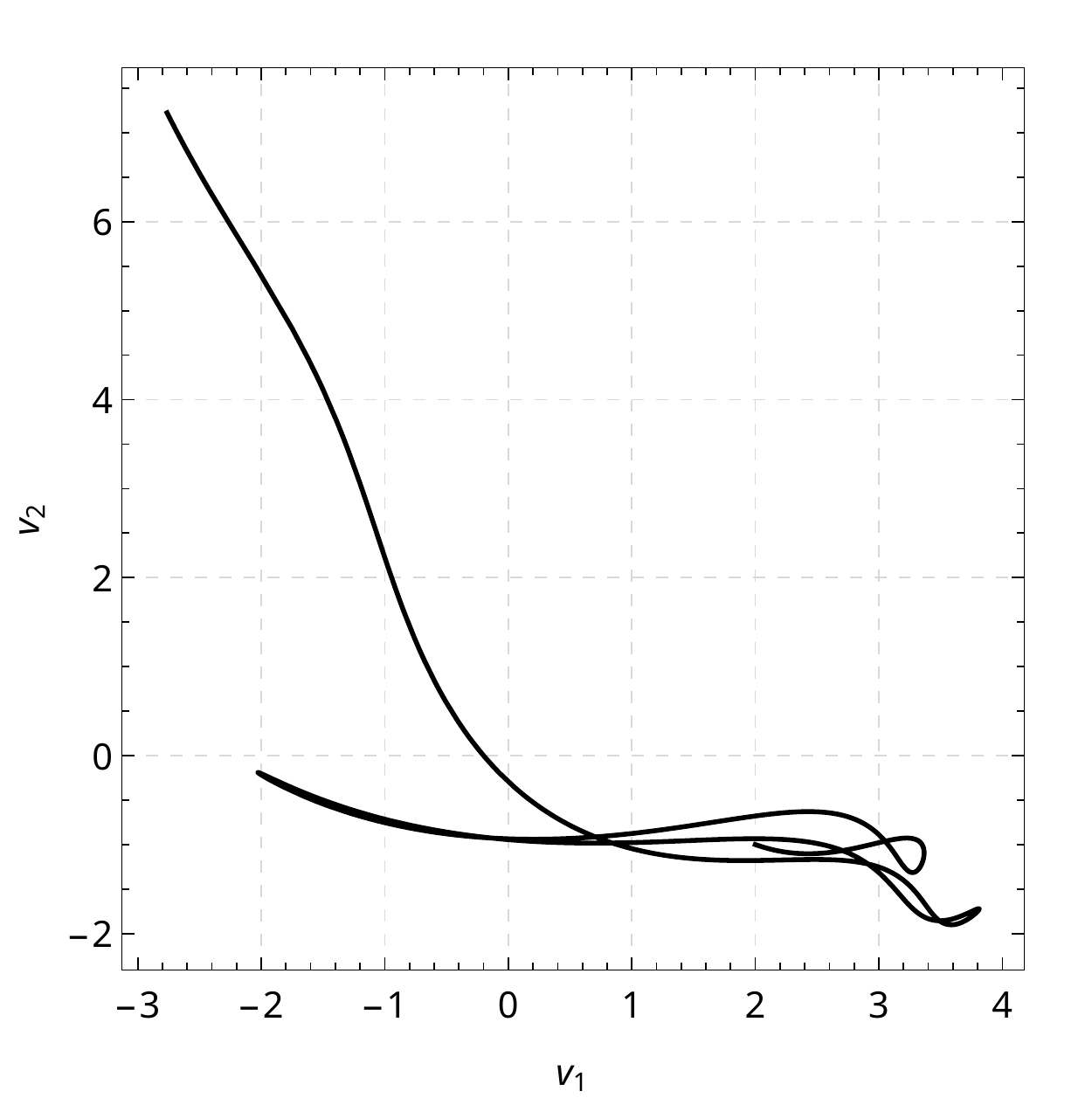}
            \caption{Parametric plot of \(v\) per example 2.}
        \end{figure}

        \begin{figure}[h]
            \subfloat[]{\includegraphics[width=0.45\textwidth]{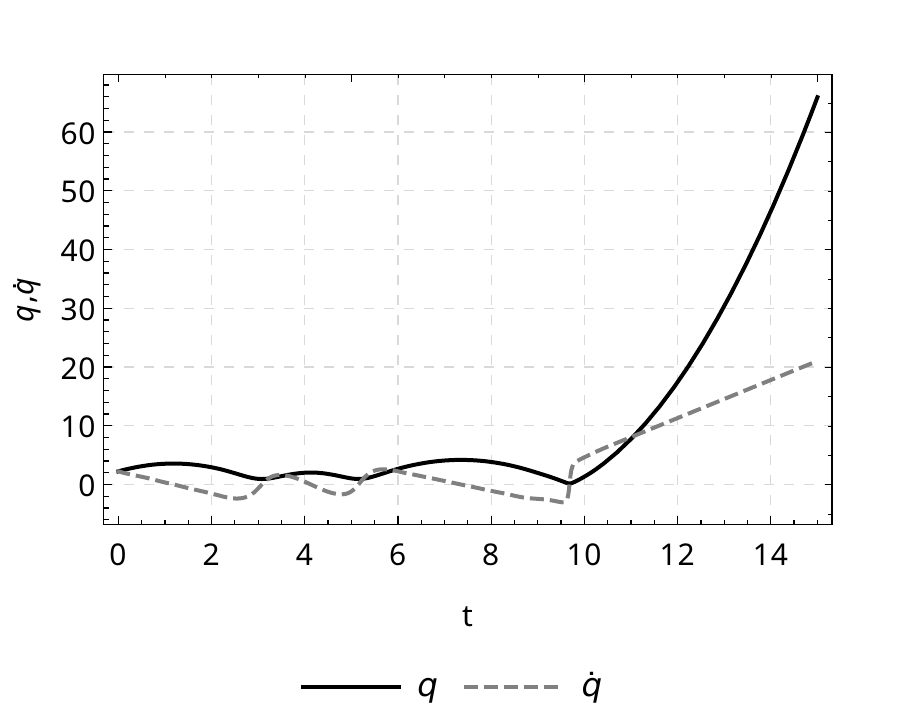}}
            \subfloat[]{\includegraphics[width=0.45\textwidth]{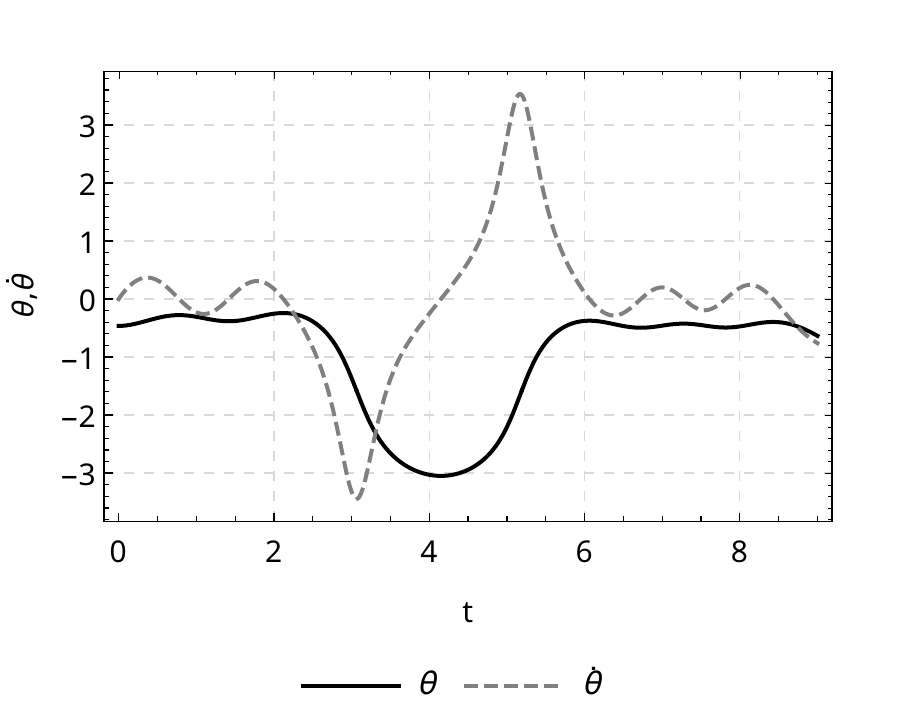}}
            \caption{Radial and angular components of  \(v\) per example 2.}
        \end{figure}

        \begin{figure}[h]
            \centering
            \includegraphics[width=0.3\textwidth]{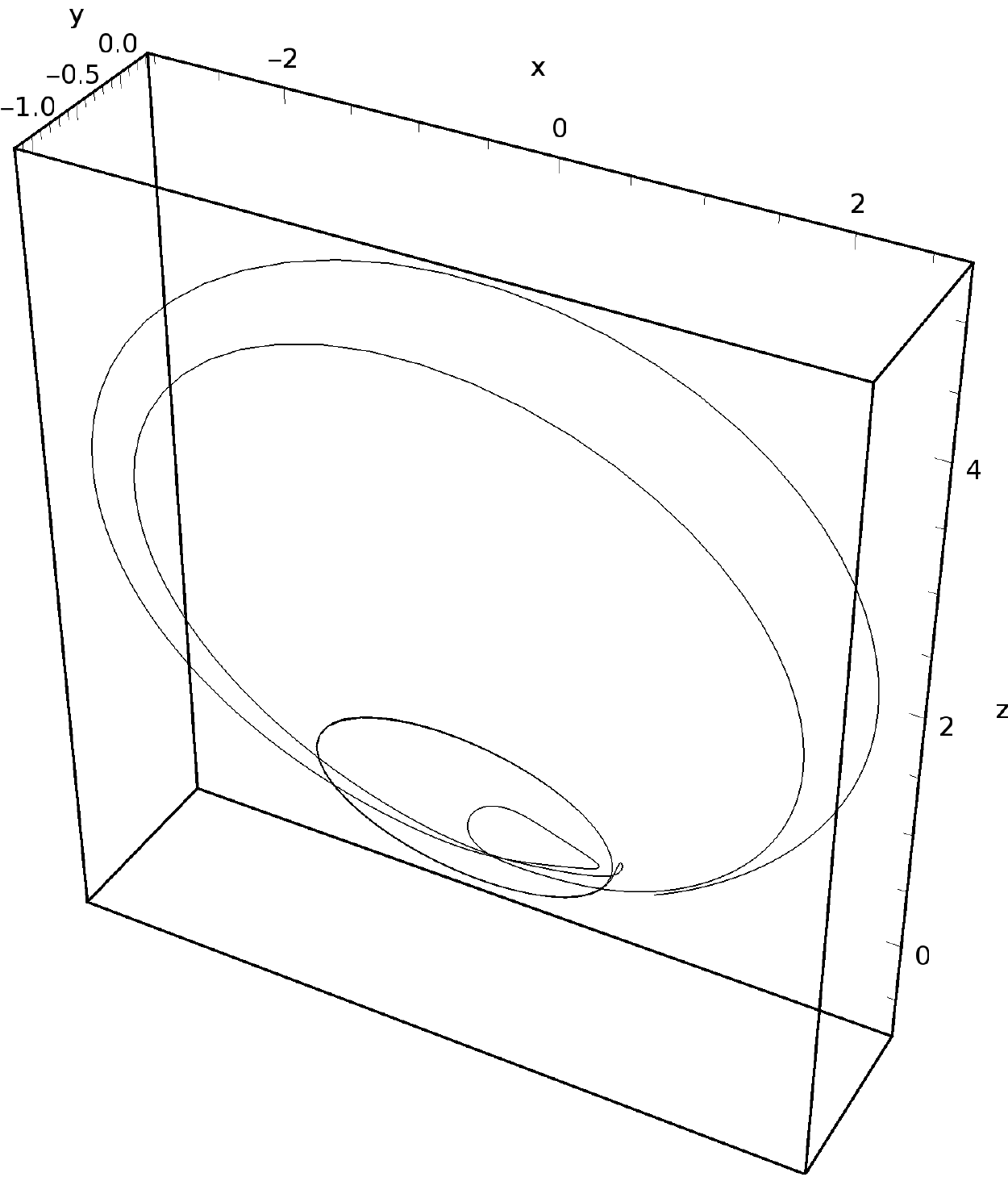}
            \caption{Components of \( \widetilde{x} \) projected into \( \mathbb{R}^3\) per example 2.}
        \end{figure}

    \newpage
    \bibliography{references.bib}

\end{document}